\documentclass[11pt, reqno]{amsart}
\usepackage{graphicx, amssymb, amsmath, amsthm}
\usepackage[utf8]{inputenc}
\usepackage{epsfig}
\usepackage{hyperref}
\usepackage{comment}
\usepackage{mathrsfs}
\numberwithin{equation}{section}
\usepackage{subfigure}
\usepackage{tikz}
\usepackage{here}
\usepackage{float}
\usepackage{pifont}
\usepackage{enumitem}
\usetikzlibrary{matrix,arrows}
\usetikzlibrary{shapes}
\usetikzlibrary{calc}
\usetikzlibrary{arrows}
\usetikzlibrary{decorations.pathreplacing,decorations.markings}
\usepackage[all]{xy}

\usepackage{tikz-cd}

\DeclareMathOperator{\Ric}{Ric}

\DeclareMathOperator{\Lip}{Lip}

\newcommand{\sff}{{\rm II}}

\newcommand{\B}{\mathbb{B}}

\newcommand{\R}{\mathbb{R}}
\renewcommand{\S}{\mathbb{S}}
\newcommand{\T}{\mathbb{T}}

\newcommand{\cA}{\mathcal A}

\newcommand{\cH}{\mathcal H}

\newcommand{\cQ}{\mathcal Q}

\newtheorem{theorem}{Theorem}[section]
\newtheorem{lemma}[theorem]{Lemma}
\newtheorem{corollary}[theorem]{Corollary}
\newtheorem{conjecture}[theorem]{Conjecture}
\newtheorem{proposition}[theorem]{Proposition}

\theoremstyle{definition}
\newtheorem{definition}[theorem]{Definition}
\newtheorem{remark}[theorem]{Remark}

\makeatletter
\newcommand{\Extend}[5]{\ext@arrow0099{\arrowfill@#1#2#3}{#4}{#5}}
\makeatother

\begin{document}

\title[Optimal quadratic decay]{
Optimal decay constant for complete manifolds of positive scalar curvature with quadratic decay}

\author[Shuli Chen]{Shuli Chen}
\address[Shuli Chen]{Department of Mathematics, University of Chicago, 5734 S University Ave, Chicago IL, 60637, United States}
\email{shulichen@uchicago.edu}

\makeatletter 
\@namedef{subjclassname@2020}{\textup{2020} Mathematics Subject Classification} 
\makeatother
\subjclass[2020]{Primary 53C23; Secondary 53C21}

\begin{abstract}
We prove that if an orientable 3-manifold $M$ admits a complete Riemannian metric whose scalar curvature is positive and has at most $C$-quadratic decay at infinity for some $C > \frac{2}{3}$, then it decomposes as a (possibly infinite) connected sum of spherical manifolds and $\mathbb{S}^2\times \mathbb{S}^1$ summands. Consequently, $M$ carries a complete Riemannian metric of uniformly positive scalar curvature. The decay constant $\frac{2}{3}$ is sharp, as demonstrated by metrics on $\mathbb{R}^2 \times \mathbb{S}^1$. This improves a result of Balacheff, Gil Moreno de Mora Sard{\`a}, and Sabourau, and partially answers a conjecture of Gromov. The main tool is a new exhaustion result using $\mu$-bubbles.

In dimensions $n = 4, 5$, we further extend results of Chodosh--Maximo--Mukherjee and Sweeney, and obtain topological obstructions to the existence of a complete Riemannian metric whose scalar curvature is positive and has at most $C$-quadratic decay at infinity for some $C > \frac{n-1}{n}$ on certain noncompact contractible $n$-manifolds. 
\end{abstract}

\maketitle

\section{Introduction}\label{Introduction}
The scalar curvature $R$, defined as the average of the sectional curvatures up to a multiplicative constant, is the weakest curvature invariant of a Riemannian metric.
On the other hand, the existence of positive scalar curvature can still place strong topological constraints on the underlying topology.
In the study of three-dimensional manifolds, a fundamental question is about understanding the topological structure of 3-manifolds that admit a Riemannian metric of positive scalar curvature $R > 0$.

In the closed case, the classical Kneser--Milnor prime decomposition theorem \cite{kneser1929geschlossene, milnor1962unique} combined with Perelman's resolution of the Poincar\'e conjecture \cite{Perelman2002,Perelman2003a,Perelman2003b} implies that every closed orientable $3$-manifold $M$ can be uniquely decomposed into prime factors as
$$M = X_1 \# \dots X_\ell \# (\S^2 \times \S^1) \# \dots \# (\S^2 \times \S^1) \# K_1 \# \dots \# K_m,$$
where each $X_i$ is a spherical manifold and each $K_k$ is an aspherical manifold. Recall that a closed 3-manifold $M$ is \emph{prime} if $M = M' \# M''$ implies that either $M'$ or $M''$ is diffeomorphic to $\S^3$. Here, a closed 3-manifold is \emph{spherical} if it is of the form $\S^3/\Gamma_i$ obtained as the quotient of the 3-sphere by a subgroup $\Gamma_i < O(4)$ of isometries acting freely on $\S^3$, and a manifold is \emph{aspherical} if it has contractible universal cover.
Perelman \cite{Perelman2002,Perelman2003a,Perelman2003b}, completing pioneering work of Schoen and Yau \cite{SY1979existence} and Gromov and Lawson \cite{gromov1980spin,gromov1980classification,gromov1983positive}, showed that a closed orientable 3-manifold admits a Riemannian metric with positive scalar curvature if and only if it decomposes as a connected sum of only spherical manifolds and $\S^2 \times \S^1$ summands.

In contrast, for open 3-manifolds, the topological structure is much more complicated, and naive generalizations of Kneser--Milnor prime decomposition do not hold in general \cite{scott1977fundamental,maillot2008some}. Nevertheless, if we strengthen the curvature condition from positive scalar curvature $R > 0$ to uniformly positive scalar curvature $R \ge 1$, then a similar decomposition theorem has recently been proved for open 3-manifolds admitting complete Riemannian metrics of uniformly positive scalar curvature. Namely, such manifolds decompose as a possibly infinite connected sum of spherical manifolds and $\S^2 \times \S^1$ summands. In this direction, Chang, Weinberger, and Yu \cite{chang2010taming} proved the decomposition for manifolds with finitely generated fundamental group using K-theory methods. Bessi\`eres, Besson, and Maillot \cite{bessieres2011ricci} showed the result for manifolds with bounded geometry via Ricci flow. 
Finally, Gromov \cite{Gromov2023FourLectures} and Wang
\cite{wang2022topology} independently established the general case, using $\mu$-bubble technique introduced by Gromov \cite{gromov1996positive}.

One may wish to consider a weaker scalar curvature assumption and ask for a decomposition theorem. In \cite{Balacheff2025}, Balacheff, Gil Moreno de Mora Sard{\`a}, and Sabourau extended the decomposition theorem to complete Riemannian 3-manifolds of positive scalar curvature with at most $C$-quadratic decay at infinity for some constant $C > 64 \pi^2$. Their proof is deduced from results on fill radius instead of via $\mu$-bubbles.

\begin{definition}[{\cite[Definition 1.2]{Balacheff2025}{\cite[Definition 1.3.5]{gil2025}}}]\label{de:decay_infinity}
	Let $M$ be a complete Riemannian $n$-dimensional manifold. Fix a basepoint $x \in M$, and denote by $r_x(y) = d(x,y)$ the distance function to $x$.
	The scalar curvature $R$ of $M$ has \emph{at most $C$-quadratic decay at infinity} if there exists a constant $D_0 > 0$ such that for every $y \in M$ with $r_x(y) \geq D_0$,
	\begin{equation*}
		R(y) > \frac{C}{r_x(y)^2}.
	\end{equation*}
\end{definition}
\begin{remark}

The constant $C$ in Definition \ref{de:decay_infinity} may depend on the basepoint $x$.
\footnote{This dependency was first noticed by Laurent Bessi\`{e}res and was pointed out to me by  St{\'e}phane Sabourau. See also \cite[Definition 1.3.5]{gil2025}.}
However, by replacing $C$ with any constant $C' < C$ arbitrarily close to $C$, a similar inequality holds for any other choice of basepoint by simply changing the constant $D_0$. Since in this paper we are only interested in conditions of the form ``for some $C > C_0$'' or ``for every $0 < C< C_0$" where $C_0$ is a fixed constant, the basepoint can be taken arbitrarily. Thus we will often not mention the basepoint for simplicity of the statements. 
\end{remark}

\begin{theorem}[{\cite[Theorem 1.3]{Balacheff2025}}] \label{thm:B}
	Let $M$ be a complete connected orientable Riemannian 3-manifold. Suppose that $M$ has positive scalar curvature with at most $C$-quadratic decay at infinity for some $C > 64 \pi^2$.
	Then $M$ decomposes as a possibly infinite connected sum of spherical manifolds and $\S^2 \times \S^1$ summands.
\end{theorem}

This theorem is closely related to the following conjecture of Gromov \cite[Section 3.6.1]{Gromov2023FourLectures}.
\begin{conjecture}[{Critical Rate of Decay Conjecture \cite{Gromov2023FourLectures}}] \label{conj:critical_rate}
	There exists a universal critical constant $C_n > 0$ such that the following holds. Let $M$ be an orientable $n$-manifold that admits a complete Riemannian metric of positive scalar curvature. 
	\begin{enumerate}
	\item For every $C < C_n$, there exists a complete Riemannian metric on $M$ of positive scalar curvature with at most $C$-quadratic decay at infinity.
	\item If $M$ admits a complete Riemannian metric with positive scalar curvature with $C$-quadratic decay at infinity for $C > C_n$, then $M$ admits a complete Riemannian metric with uniformly positive scalar curvature. \label{it:co:critical_rate}
	\end{enumerate}
\end{conjecture}

In this paper, we improve the decay constant in Theorem \ref{thm:B} and find the optimal one to be $\frac{2}{3}$.

\begin{theorem} \label{thm:main}
	Let $M$ be a complete connected orientable Riemannian 3-manifold. Suppose that $M$ has positive scalar curvature with at most $C$-quadratic decay at infinity for some $C > \frac{2}{3}$.
	Then $M$ decomposes as a possibly infinite connected sum of spherical manifolds and $\S^2 \times \S^1$ summands.

    The constant $\frac{2}{3}$ is optimal; the manifold $\R^2 \times \S^1$ admits a complete metric of positive scalar curvature with at most $C$-quadratic decay for every $0 < C < \frac{2}{3}$.
\end{theorem}

Note that $\R^2 \times \S^1$ does not decompose as a connected sum of spherical manifolds and $\S^2 \times \S^1$ (see e.g. \cite[Proposition 8.2]{Balacheff2025}).

The following rigidity result, which addresses the case \eqref{it:co:critical_rate} of Conjecture \ref{conj:critical_rate}, is a direct consequence of Theorem \ref{thm:main} and an adaptation of Gromov--Lawson's Surgery Theorem \cite[Theorem A]{gromov1980classification}. In particular, we find $C_3 = \frac{2}{3}$.

\begin{corollary} \label{cor:surgery_uniform_scal}
Let $M$ be an orientable 3-manifold. If $M$ admits a complete Riemannian metric of positive scalar curvature with at most $C$-quadratic decay at infinity for some $C > \frac{2}{3}$ with respect to some basepoint, it also admits a complete Riemannian metric with uniformly positive scalar curvature. The constant $C_3 = \frac{2}{3}$ is optimal here.
\end{corollary}

Unlike \cite{Balacheff2025}, our main tool is the method of $\mu$-bubbles. Our main theorem is a consequence of the following construction result and new exhaustion result. 

\begin{proposition}\label{prop:construction}
For every $n \ge 3$, the $n$-manifold $\R^2 \times \T^{n-2}$ admits a complete metric of positive scalar curvature with at most $C$-quadratic decay for every $0 < C < \frac{n-1}{n}$.    
\end{proposition}

Note that $\R^2 \times \T^{n-2}$ does not carry a complete metric of uniformly positive scalar curvature \cite[p. 648]{Gromov2018}. Thus Proposition \ref{prop:construction} implies that the critical constant $C_n$ in Conjecture \ref{conj:critical_rate} must satisfy $C_n \ge \frac{n-1}{n}$. Moreover, in view of the proposition below, we conjecture that $C_n  = \frac{n-1}{n}$. 

\begin{proposition}\label{prop:exhaustion}
Let $3 \le n \le 7$. Suppose $M$ is a complete orientable $n$-manifold. Suppose that $M$ has positive scalar curvature with at most $C$-quadratic decay at infinity for some $C > \frac{n-1}{n}$.
Then $M$ admits an exhaustion $K_1 \subset K_2 \subset \dots \subset M$ such that $M = \cup_{i=1}^\infty K_i$, each $K_i$ is a compact codimension-0 smooth submanifold with smooth boundary $\partial K_i$, and $\partial K_i$ is a compact orientable $(n-1)$-manifold that admits a metric of positive scalar curvature. 
\end{proposition}

We note that for a smooth compact exhaustion $K_1 \subset K_2 \subset \cdots \subset \R^2 \times \T^{n-2}$ of $\R^2 \times \T^{n-2}$, $3 \le n \le 7$, for sufficiently large $i$, $\partial K_i$ admits a non-zero degree map to $\T^{n-1}$ and therefore admits no metric of positive scalar curvature \cite{SY1979descent}. Thus the quadratic decay constant $\frac{n-1}{n}$ in Proposition \ref{prop:exhaustion} cannot be lowered. Proposition \ref{prop:exhaustion} extends the corresponding exhaustion result of Gromov regarding uniformly positive scalar curvature \cite[Section 3.7.2]{Gromov2023FourLectures}.  We can thereby extend certain results regarding uniformly positive scalar curvature to positive scalar scalar curvature with quadratic decay. 
We have the following important consequence of Proposition \ref{prop:exhaustion} (cf. \cite[Corollary 3.4]{chodosh2024fourmanifold}).
\begin{corollary}\label{cor:boundary-tame}
For $n \in \{3,4,6,7\}$, suppose that $W$ is a compact smooth $n$-manifold with boundary and that some component of $\partial W$ does not admit a Riemannian metric with positive scalar curvature. Then the interior of $W$ does not admit a complete Riemannian metric with at most $C$-quadratic decay at infinity for some $C > \frac{n-1}{n}$.
\end{corollary}

In dimension 4, using Corollary \ref{cor:boundary-tame}, we obtain the following extensions of Theorem 1.1 and Theorem 1.2 of \cite{chodosh2024fourmanifold}. 

\begin{theorem}\label{thm:Mazur}
Suppose that $M$ is the interior of a Mazur manifold $W$. If $M$ admits a complete Riemannian metric with at most $C$-quadratic decay at infinity for some $C > \frac{3}{4}$, then $W$ must be diffeomorphic to the $4$-ball $\B^4$ and consequently $M$ must be diffeomorphic to the standard $\R^4$.
\end{theorem}
Here, a \emph{Mazur manifold} is a compact, contractible smooth 4-manifold admitting a smooth handle decomposition with a single 0-handle, a single 1-handle and a single 2-handle.
\begin{theorem}\label{thm:b4}
Suppose that $M$ is the interior of a compact contractible smooth 4-manifold with boundary $W$. If $M$ admits a complete Riemannian metric with at most $C$-quadratic decay at infinity for some $C > \frac{3}{4}$, then $W$ must be homeomorphic to the $4$-ball $\B^4$.
\end{theorem}

In dimension 5, we also have the following extension of Theorem A of \cite{sweeney2025positive}.
\begin{theorem}\label{thm:b5}
Suppose that $M$ is the interior of a compact, contractible 5-manifold with boundary
$X,$ such that $\pi_3(X,\partial X) = 0$. If $M$ admits a complete Riemannian metric with at most $C$-quadratic decay at infinity for some $C > \frac{4}{5}$, then $M$ must be diffeomorphic to $\R^5$.
\end{theorem}

This paper is organized as follows. In Section 2, we prove Proposition \ref{prop:construction}. In Section 3, we briefly discuss the $\mu$-bubble method and prove Proposition \ref{prop:exhaustion}, Corollary \ref{cor:boundary-tame}, Theorem \ref{thm:Mazur}, Theorem \ref{thm:b4}, and Theorem \ref{thm:b5}. In Section 4, we give the proof of our main theorem, Theorem \ref{thm:main}. 

\subsection*{Acknowledgments}
The author would like to thank St{\'e}phane Sabourau for giving a talk and having discussions on the work \cite{Balacheff2025} which motivated the present paper, for pointing out the dependency of the constant $C$ on the basepoint in Definition \ref{de:decay_infinity}, and for sharing reference \cite{gil2025}. She is grateful to Otis Chodosh, Yujie Wu, Jianchun Chu and Jintian Zhu for their interest, inspiring discussions, and helpful suggestions. She also thanks Paul Sweeney for answering questions on the work \cite{sweeney2025positive}. Furthermore, she thanks Diego Guajardo for carefully reading the previous version of the paper and pointing out typos and inaccuracies.

\section{Proof of Proposition \ref{prop:construction}}
In this section, we construct the desired metric in Proposition \ref{prop:construction}. Let $n \ge 3$. Let $c >0$ and $0 <r_0 < \frac{\pi}{2}$ be constants to be chosen later.

For $\R^2 \times \T^{n-2}$, we parametrize $\R^2$ by polar coordinate $(r, \theta)$ and let $g_{\T^{n-2}}$ be a flat metric on $\T^{n-2}$. 
We will construct the desired metric on $\R^2 \times \T^{n-2}$ by gluing together two positive scalar curvature metrics. We fix a basepoint $x$ which has $r$-coordinate equals $2r_0$.

For $r \ge r_0$, we consider a metric of the form $g_+ = dr^2 + c^2 r^{2\alpha} d\theta^2 + c^2 r^{2\alpha} g_{\T^{n-2}}$ on the end. 
Then the scalar curvature of this warped product metric with warping function $f(r) = cr^\alpha$ is given by (see e.g. \cite[p. 157]{gromov1983positive})
\begin{align*}
R &= -\frac{(n-1)(n-2)(f')^2}{f} - \frac{2(n-1)f''}{f} \\
&= -\frac{(n-1)(n-2)\alpha^2}{r^2} - \frac{2(n-1)\alpha(\alpha-1)}{r^2} \\
& = \frac{-n(n-1)\alpha^2 + 2(n-1) \alpha}{r^2} \\
& = \frac{1}{r^2}\left(-n(n-1)\left(\alpha - \frac{1}{n}\right)^2 + \frac{n-1}{n} \right) \\
&  \le \frac{n-1}{n}\frac{1}{r^2}.
\end{align*}
Thus on $[r_0, +\infty) \times \S^1 \times \T^{n-2}$, we can get positive scalar curvature with at most $C$-quadratic decay at infinity for the metric for any $0 < C < \frac{n-1}{n}$ by choosing corresponding $\alpha \in (0, \frac{1}{n}]$.

Under this metric, the mean curvature of the slice $\{r_0\} \times \S^1 \times \T^{n-2}$ with respect to $\partial r$ is given by
$$H_+ = \frac{(n-1)f'(r_0)}{f(r_0)} = \frac{(n-1)\alpha}{r_0} \le  \frac{(n-1)}{nr_0}.$$

For $r \le 2r_0$, we consider the metric $g_- = d r^2 + \sin^2 (r) d\theta^2 + \sin^2(r_0) g_{\T^{n-2}}$. Since $d r^2 + \sin^2 (r) d\theta^2$ is the canonical metric on $\S^2$, $g_-$ is smooth at $r=0$, and the scalar curvature is $2$.
The mean curvature of the slice $\{r_0\} \times \S^1 \times \T^{n-2}$ with respect to $\partial r$ is given by
$$H_- = \frac{\cos(r_0)}{\sin(r_0)}.$$

Notice that for $r_0$ sufficiently small, we have $H_- > H_+$. Thus we choose $r_0$ sufficiently small and let $c = \dfrac{\sin r_0}{r_0^\alpha}$. 

Then the two metrics $g_+$ and $g_-$ agree at $r=r_0$, and we have $H_- > H_+$. Since $H_- > H_+$, if we glue together $g_-$ and $g_+$, the scalar curvature is positive in the distributional sense (see \cite[Section 2]{Miao2002}), so we expect to be able to smooth out the singularity at $r_0$. 

Indeed, consider the manifold with boundary $[r_0, +\infty) \times \S^1 \times \T^{n-2}$. By Theorem 5 of \cite{Brendle2011} (which is applicable here because the construction is purely local; also note that they use the opposite sign convention for mean curvature), since $H_- > H_+$, we can find a positive scalar curvature metric $\hat g$ on $[r_0, +\infty) \times \S^1 \times \T^{n-2}$ such that $\hat g$ agrees with $g_+$ for $r \ge 2r_0$, and $\hat g$ agrees with $g_-$ in a neighborhood of $\{r_0\} \times \S^1 \times \T^{n-2}$. 

Then we can extend $\hat g$ to a smooth metric on the whole $\R^2 \times \T^{n-2}$ by gluing with $g_-$ for $r \le r_0$. The metric $\hat g$ is the desired complete metric with positive scalar curvature that has at most $C$-quadratic decay for the given $0 < C < \frac{n-1}{n}$ with respect to basepoint $x$.

\section{Exhaustion via $\mu$-bubbles}
We first recall general existence and stability results for $\mu$-bubbles. 

For $3 \le n\leq 7$, consider $(M,g)$, a Riemannian $n$-manifold with boundary. Assume that $\partial M = \partial_- M \cup \partial_+ M$ is a labeling of the boundary components so that each of the sets $\partial_\pm M$ is non-empty. Fix a smooth function $h$ on $\mathring M$ with $h\to \pm \infty$ on $\partial_\pm M$. 
Choose a Caccioppoli set $\Omega_0$ with smooth boundary such that $\Omega_0$ contains an open neighborhood of $\partial_+ M$ and is disjoint from $\partial_- M$.

For all Caccioppoli sets $\Omega$ in $M$ with $\Omega\Delta \Omega_0\Subset \mathring M$, we consider the following $\mu$-bubble functional
\begin{equation}\label{problem.variation}
\cA(\Omega)=\cH^{n-1}(\partial^* \Omega \cap \mathring M) - \int_M (\chi_\Omega-\chi_{\Omega_0})h \, d\cH^n,
\end{equation}
where $\partial^* \Omega$ denotes the reduced boundary of $\Omega$. We will call a Caccioppoli set $\Omega$ minimizing $\cA$ in this class a \emph{$\mu$-bubble}. 

The next proposition shows $\mu$-bubbles are
easier to construct than stable minimal hypersurfaces and enjoy the same regularity.
\begin{proposition}[{\cite[Proposition 2.1]{zhu2021width}}{\cite[Proposition 12]{chodosh2024soapbubble}}]\label{prop: existence.regularity}
For $3 \le n \le 7$, there exists a smooth minimizer $\Omega$ for $\cA$ such that $\Omega\Delta \Omega_0$ is compactly contained in the interior of $M$.
\end{proposition}

We next discuss the first and second variation for a $\mu$-bubble.
\begin{lemma}[{\cite[Lemma 13]{chodosh2024soapbubble}}]\label{lem:1st-var}
If $\Omega_t$ is a smooth $1$-parameter family of regions with $\Omega_0 = \Omega$ and normal speed $\psi$ at $t=0$, then 
\[\frac{d}{dt}\cA (\Omega_t)=\int_{\Sigma_t} (H-h)\psi  \, d\cH^{n-1}\]
where $H$ is the scalar mean curvature of $\Sigma_t = \partial\Omega_t \cap \mathring M$. 
In particular, a $\mu$-bubble $\Omega$ satisfies \[
H = h
\]
along $\partial\Omega \cap \mathring M$.
\end{lemma}
\begin{lemma}\label{lem:2nd-var}
Consider a $\mu$-bubble $\Omega$ with $\partial\Omega \cap \mathring M = \Sigma$.
Assume that $\Omega_t$ is a smooth $1$-parameter family of regions with $\Omega_0 = \Omega$ and normal speed $\psi$ at $t=0$, then $\cQ(\psi):=\frac{d^2}{dt^2}\big|_{t=0}(\cA(\Omega_t))\ge 0$ where $\cQ(\psi)$ satisfies 
\begin{align*}
\cQ(\psi) 
& = \int_\Sigma \left(|D_\Sigma \psi|^2  - \big(|\sff|^2 + \Ric_{M}(\nu,\nu) + \langle D_{M} h, \nu \rangle\big)\psi^2  
\right)d\cH^{n-1} \\
& \le \int_\Sigma \left(|D_\Sigma \psi|^2  + \frac{1}{2}\big(R_\Sigma - R_M - \frac{n}{n-1}h^2 + 2|D_M h|\big)\psi^2  
\right)d\cH^{n-1}.
\end{align*}
where $D_\Sigma \psi$ is the gradient of $\psi$ on $\Sigma$, $D_M h$ is the gradient of $h$ on $M$, $\sff$ is the second fundamental form of $\Sigma$, and $\nu$ is the outwards pointing unit normal.
\end{lemma}
\begin{proof}
The second variation formula is given by (see e.g. {\cite[Lemma 3.3]{chen2024geroch}})
\begin{align*}
\cQ(\psi) 
= \int_\Sigma \left(|D_\Sigma \psi|^2  - \big(|\sff|^2 + \Ric_{M}(\nu,\nu) + \langle D_{M} h, \nu \rangle\big)\psi^2  
\right)d\cH^{n-1},
\end{align*}  
By Schoen--Yau's rearrangement trick \cite[p. 165]{SY1979descent}, the traced Gauss equation gives 
$$R_{M} = R_\Sigma + 2  \operatorname{Ric}_{M}(\nu,\nu)+ |\sff|^2 - H^2.$$

Combining with the inequalities $|\sff|^2 \ge \frac{H^2}{n-1}$ and $-\langle D_{M} h, \nu \rangle \le |D_{M} h|$, we get
\begin{align*}
\cQ(\psi) 
\le \int_\Sigma \left(|D_\Sigma \psi|^2  + \frac{1}{2}\big(R_\Sigma - R_M - \frac{n}{n-1}h^2 + 2|D_M h|\big)\psi^2  
\right)d\cH^{n-1},
\end{align*} 
where we used $H = h$ on $\Sigma$.
\end{proof}

We also need the following result of solving a Riccati equation, which will be useful for finding the appropriate function $h$ in the $\mu$-bubble functional.
\begin{lemma}
Let $C>0$ be a constant. Consider the following ODE for $r >0$: 

\begin{equation}\label{eqn:Riccati}
2 f' - \frac{n}{n-1}f^2 - \frac{C}{r^2} = 0.   
\end{equation}

If $C > \frac{n-1}{n}$, then
\begin{equation}\label{eqn:Riccati_soln}
f(r) = \frac{2(n - 1)}{n r} \left[
-\frac{1}{2}
+ \mu \cdot \tan(\mu \log r)
\right]  
\end{equation}
is a solution to \eqref{eqn:Riccati},
where
$$
\mu = \frac{1}{2} \sqrt{ \frac{n C}{n - 1} - 1 }.
$$
\end{lemma}
\begin{proof}
We can easily verify \eqref{eqn:Riccati_soln} is a solution to \eqref{eqn:Riccati} by direct computation. Here we show how to solve this Riccati equation for completeness. 

We put the ODE \eqref{eqn:Riccati} into standard form
$$f' - \frac{n}{2(n-1)}f^2 - \frac{C}{2r^2} = 0$$  

By \cite[Theorem 2.1 of Chapter One]{reid1972Riccati}, equation \eqref{eqn:Riccati} has a solution on an interval $I$ if and only if there is a solution $u(r)$ to the second-order linear ODE
\begin{equation}\label{eqn:second_order}
u'' + \frac{nC}{4(n-1)r^2}u = 0    
\end{equation}
such that on $I$, we have $u \neq 0$ and $$f  = - \frac{2(n-1)}{n}\frac{u'}{u}.$$ 

We try solution to \eqref{eqn:second_order} of the form $u = r^\lambda$, and we get the indicial equation
$$\lambda^2 - \lambda + \frac{nC}{4(n-1)} = 0.$$
Since $C > \frac{n-1}{n}$, the discriminant is $\Delta = 1 - \frac{nC}{(n-1)} < 0$. Let $\mu = \frac{1}{2} \sqrt{ -\Delta} = \frac{1}{2} \sqrt{ \frac{n C}{n - 1} - 1}$. Then the roots to the indicial equation are $\lambda = \frac{1}{2} \pm i\mu$, and the general solution of $u$ is given by 
$$u(x) = r^{1/2}(C_1 \cos(\mu\log r)+C_2\sin(\mu\log r))$$
where $C_1, C_2$ are arbitrary constants.

Thus the general solution of $f$ is given by 

$$f  = - \frac{2(n-1)}{n}\frac{u'}{u} = -\frac{2(n - 1)}{n r} \left[
\frac{1}{2}
+ \mu \cdot \frac{ -C_1 \sin(\mu \log r) + C_2 \cos(\mu \log r) }{ C_1 \cos(\mu \log r) + C_2 \sin(\mu \log r) }
\right].$$ 

Choosing $C_1 = 1$, $C_2 =0$ gives \eqref{eqn:Riccati_soln}.
\end{proof}

We are now ready to prove the exhaustion result.

\begin{proof}[Proof of Proposition \ref{prop:exhaustion}]
Fix a basepoint $x_0 \in M$, and denote by $r(y) = d(x_0,y)$ the distance function to $x_0$. 

Then the scalar curvature assumption implies that  there exists $D_0 >0$ such that if $r(y) > D_0$, then $R(y) > \frac{C}{r(y)^2}.$

Let $\tilde{C}$ be such that $\frac{n-1}{n}< \tilde{C} < C$. 
Let $\tilde{\mu} = \frac{1}{2} \sqrt{ \frac{n \tilde{C}}{n - 1} - 1}$. Choose $0 < \varepsilon < 1$ be such that 
$$
\frac{(1+\varepsilon)^2 D_0^2}{(D_0 - \varepsilon)^2} < \frac{C}{\tilde{C}}.    
$$

Then for all $r \ge D_0$, we have
$$
\frac{(1+\varepsilon)^2 r^2}{(r - \varepsilon)^2} < \frac{C}{\tilde{C}},   
$$
so 
\begin{equation}\label{eqn:epsilon}
\frac{\tilde{C}(1+\varepsilon)^2 }{(r- \varepsilon)^2} < \frac{C}{r^2}.    
\end{equation}

By \cite[Section 2]{greene1979c}, we can find a smoothing $\rho$ of $r$ on $M$ such that $|\rho(x) - r(x)| < \varepsilon$ for all $x \in M$, and $\Lip \rho < 1 + \varepsilon$. By perturbing we further assume that for all integers $k > 0$, $(1+\varepsilon)\exp({\tilde \mu}^{-1}(k \pi - \frac{\pi}{2}))$ is a regular value of $\rho$.

Let $f(r) = \frac{2(n - 1)}{n r} \left[
-\frac{1}{2}
+ \tilde{\mu} \cdot \tan(\tilde{\mu} \log r)
\right]$ as in \eqref{eqn:Riccati_soln}, and define $h(x) = - f(\frac{\rho(x)}{1+\varepsilon})$.

We compute that for $r, \rho > D_0$,
\begin{align*}
& - 2|D_M h| + \frac{n}{n-1}h^2 + R_M \\
= & -2 |f'(\frac{\rho}{1+\varepsilon})| \cdot\frac{|D_M \rho|}{1+\varepsilon} + \frac{n}{n-1}f^2(\frac{\rho}{1+\varepsilon}) + R_M \\
> & -2 |f'(\frac{\rho}{1+\varepsilon})|  + \frac{n}{n-1}f^2(\frac{\rho}{1+\varepsilon}) + \frac{C}{r^2} \\
= & -\frac{\tilde{C}}{\rho^2/(1+\varepsilon)^2} + \frac{C}{r^2} &\text{by \eqref{eqn:Riccati}}\\
> & -\frac{\tilde{C}(1+\varepsilon)^2}{(r-\varepsilon)^2} + \frac{C}{r^2} \\
> & \, 0 & \text{by \eqref{eqn:epsilon}.}
\end{align*}

Take $k_1> 0$ to be an integer such that $\exp({\tilde \mu}^{-1}(k_1 \pi - \frac{\pi}{2})) > D_0 + 1$.
We define $M_0 = \{x \in M  \mid  \rho (x) < (1+\varepsilon)\exp({\tilde \mu}^{-1}(k_1 \pi - \frac{\pi}{2}))\}$.
Let $M_1 = \{x \in M  \mid  (1+\varepsilon)\exp({\tilde \mu}^{-1}(k_1 \pi - \frac{\pi}{2})) \le \rho (x) \le (1+\varepsilon)\exp({\tilde \mu}^{-1}(k_1 \pi + \frac{\pi}{2}))\}$, with boundaries $\partial_-M_1 = \{x \in M \mid \rho (x) = \exp({\tilde \mu}^{-1}(k_1 \pi + \frac{\pi}{2}))\}$ and $\partial_+M_1 = \{x \in M \mid \rho (x) = (1+\varepsilon)\exp({\tilde \mu}^{-1}(k_1 \pi - \frac{\pi}{2}))\}$.
Then $h \to \pm \infty $ on $\partial_\pm M_1$.
Notice that on $M_1$, $\rho > D_0$ and $r> D_0$.

Let $L_1 \in ((1+\varepsilon)\exp({\tilde \mu}^{-1}(k_1 \pi - \frac{\pi}{2})), (1+\varepsilon)\exp({\tilde \mu}^{-1}(k_1 \pi + \frac{\pi}{2})))$ be a regular value of $\rho$. We fix a smooth reference region $\hat{\Omega}_1 = \{x \in M  \mid  (1+\varepsilon)\exp({\tilde \mu}^{-1}(k_1 \pi - \frac{\pi}{2})) \le \rho (x) \le L_1\} $, and consider the $\mu$-bubble functional \eqref{problem.variation} on $M_1$ with $h$ as defined above.

Then Proposition \ref{prop: existence.regularity} produces a smooth minimizer $\Omega_1 \subset M_1$ for \eqref{problem.variation} such that $\Omega_1\Delta \hat{\Omega}_1$ is compactly contained in the interior of $M_1$. Then $\Omega_1$ contains a neighborhood of $\partial_+ M_1$ in $M_1$. Let $\Sigma_1 = \partial \Omega_1\cap \mathring M_1$. Note that $\Sigma_1$ is orientable because $M$ is. 
Let $K_1 = M_0 \cup \Omega_1$. Then $K_1$ is a smooth compact region and $\partial K_1 = \Sigma_1$.

Then using the stability of $\Sigma_1$, we get that for all nonzero smooth function $\psi$ on $\Sigma_1$, we have
\begin{align*}
& \int_{\Sigma_1} (|D_{\Sigma_1} \psi|^2  + \frac{1}{2}R_{\Sigma_1} \psi^2) \\
 \ge & \int_{\Sigma_1} (- 2|D_M h| + \frac{n}{n-1}h^2 + R_M )\psi^2 &\text{by Lemma \ref{lem:2nd-var}}\\
 > & \, 0 &\text{by above computation.}
\end{align*}
We now show $\Sigma_1$ admits a metric of positive scalar curvature. When $n=3$, $\dim\Sigma_1 =2$, so $R_{\Sigma_1} = 2K_{\Sigma_1}$ is the Gaussian curvature and we can take $\psi=1$ on any connected component $\Sigma_1'\subset \Sigma_1$ to find
\[
2\pi\chi(\Sigma_1') = \int_{\Sigma_1'} K_{\Sigma_1} > 0,
\]
so $\Sigma_1'$ is a sphere and thus admits positive scalar curvature. 

When $n\geq 4$ we observe that
\[
\frac{2(n-2)}{n-3} \geq 1
\]
so using $|D_{\Sigma_1} \psi|^2\geq 0$ we find that 
\[
\int_{\Sigma_1} \frac{2(n-2)}{n-3}|D_{\Sigma_1} \psi|^2 +\frac 12 R_{\Sigma_1} \psi^2 > 0,
\]
i.e. the conformal Laplacian
\[
L_{\Sigma_1} := -\frac{4(n-2)}{n-3}\Delta_{\Sigma_1} +  R_{\Sigma_1}
\]
is a positive operator (e.g.\ all eigenvalues are positive). Letting $u>0$ denote the first eigenfunction of $L_{\Sigma_1}$ (so $L_{\Sigma_1} u = \lambda u > 0$), a computation (see p.\ 9 in \cite{SY1979descent}) implies that the metric $\tilde g_{\Sigma_1} = u^{\frac{4}{n-3}}g_{\Sigma_1}$ on $\Sigma_1$ has scalar curvature
\[
\tilde R_{\Sigma_1} = u^{\frac{n+1}{n-3}}L_{\Sigma_1}u > 0. 
\]
This again shows $\Sigma_1$ admits positive scalar curvature metric.

Lastly, assume that $K_1 \subset \dots \subset K_{i-1}$ have been chosen as in the statement of the theorem. Let $k_i = k_1 + 2(i-1)$. Then we can define $M_i$, $\Omega_i$, and $\Sigma_i$ the same way as $M_1$ $\Omega_1$, and $\Sigma_1$. We let $K_i = \{x \in M \mid \rho(x) < (1+\varepsilon)\exp(\tilde{\mu}^{-1}(k_i \pi - \frac{\pi}{2})) \} \cup \Omega_i$.
The same computation shows $\Sigma_i$ also admits a metric of positive scalar curvature. 
    
\end{proof}

Other results then follow from Proposition \ref{prop:exhaustion}.
\begin{proof}[Proof of Corollary \ref{cor:boundary-tame}]
Let $M$ denote the interior of $W$. The collar neighborhood theorem yields a neighborhood of infinity $U$ diffeomorphic to $\partial W \times \R$. For $i$ sufficiently large, the exhaustion of $M$ obtained in Proposition \ref{prop:exhaustion} will have that $\partial K_i$ is a separating hypersurface in $U$ and $\partial K_i$ admits a Riemannian metric with positive scalar curvature. The assertion follows by combining Proposition \ref{prop:exhaustion} and Proposition 2.4 of \cite{chodosh2024fourmanifold}.
\end{proof}

Theorems \ref{thm:Mazur} and \ref{thm:b4} are now consequences of Corollary \ref{cor:boundary-tame}.

\begin{proof}[Proof of Theorems \ref{thm:Mazur} and \ref{thm:b4}]
We follow the proof of Theorem 1.1 and Theorem 1.2 of \cite{chodosh2024fourmanifold}. We use Corollary \ref{cor:boundary-tame} to replace \cite[Corollary 3.4]{chodosh2024fourmanifold} and get that $\partial W$ admits positive scalar curvature. The rest are exactly the same as in \cite{chodosh2024fourmanifold}.
\end{proof}

Before we prove Theorem \ref{thm:b5}, we note that Corollary \ref{cor:boundary-tame} doesn't apply for $n=5$. Therefore we need another lemma.

\begin{lemma}[cf. {\cite[Corollary 3.17]{sweeney2025positive}}]
\label{lem:aspherical}
Let $X^{n+1}$, $n \in \{4, 5\}$, be a compact, contractible $n$-manifold with boundary such that the interior of $X$ admits a complete Riemannian metric with at most $C$-quadratic decay at infinity for some $C > \frac{n}{n+1}$. If $n = 4$, further assume $\pi_3(X, \partial X) = 0$, and if $n = 5$, further assume $\pi_3(X, \partial X) = 0$ and $\pi_4(X, \partial X) = 0$. Then the boundary $\partial X$ has a finite cover that is homotopy equivalent to $\S^n$ or a connected sum of finitely many copies of $\S^{n-1} \times \S^1$.
\end{lemma}
\begin{proof}
Since $X^{n+1}$ is a compact, contractible $(n+1)$-manifold with boundary, we have that $\partial X$ is an integral homology sphere by \cite[Proposition 3.1]{sweeney2025positive}. Therefore $\partial X$ is connected and orientable.
The collar neighborhood theorem yields a neighborhood of infinity $U$ diffeomorphic to $\partial W \times \R$. For $i$ sufficiently large, the exhaustion of $X$ obtained in Proposition \ref{prop:exhaustion} will have that $\partial K_i$ is a separating hypersurface in $U$ and $\partial K_i$ admits a Riemannian metric with positive scalar curvature. 

Consider now the projection map $\pi:U\to \partial X$ and its restriction $\pi|_{\partial K_i}:\partial K_i\to \partial X$. If $\omega$ is the volume form on $\partial X$, then $\int_{\partial K_i}(\pi|_{\partial K_i})^*\omega\neq 0$ and so $\pi|_{\partial K_i}$ has non-zero degree.
   
By \cite[Lemma 3.4]{sweeney2025positive}, for $n=4$, then contractible $X$ with $\pi_3(X, \partial X) = 0$ implies $\pi_2(\partial X) = 0$; if $n = 5$, then contractible $X$ with $\pi_3(X, \partial X) =\pi_4(X, \partial X) = 0$ implies $\pi_2(\partial X) = \pi_3(\partial X)=0$. Now the conclusion follows from Theorem 2 of \cite{chodosh2023classifying}.
\end{proof}

\begin{proof}[Proof of Theorem \ref{thm:b5}]
We follow the proof of Theorem A in \cite{sweeney2025positive} and use Lemma \ref{lem:aspherical} to replace \cite[Corollary 3.17]{sweeney2025positive}. The rest are exactly the same.   
\end{proof}

\section{Proof of the decomposition theorem}
Before we prove Theorem \ref{thm:main}, we recall the notion of infinite connected sum modeled on a locally finite graph,
following \cite{scott1977fundamental, wang2022topology,Balacheff2025}. 

\begin{definition}
	Let $\mathcal{F}$ be a family of connected 3-manifolds. A 3-manifold $M$ decomposes as a \emph{connected sum} of members of $\mathcal{F}$ modeled on a locally finite graph $\mathcal{G}$ if there is a map assigning to each vertex~$v$ a copy $M_v$ of a manifold in $\mathcal{F}$ such that $M$ is diffeomorphic to the manifold constructed as follows:

	\begin{enumerate}
		\item For each vertex $v$, construct a new manifold $Y_v$ by removing from $M_v$ a number of $\deg{(v)}$ 3-balls from its interior,
		\item For each edge $e$ joining two vertices $v$ and $v'$, glue two spherical boundary components from $\partial Y_v$ and $\partial Y_{v'}$ along an orientation-reversing diffeomorphism.
	\end{enumerate}
\end{definition}
Note that if a 3-manifold decomposes as an infinite connected sum, the graph $\mathcal{G}$ on which $M$ is modeled may not be unique.

\begin{proof}[Proof of Theorem \ref{thm:main}]
By Proposition \ref{prop:exhaustion}, consider an exhaustion of $M$ by compact domains $K_1 \subset K_2 \subset \dots \subset M$, whose boundaries $\partial K_i$ form a locally finite collection of orientable closed surfaces that admit metrics of positive scalar curvature. Denote the connected components of the boundary surfaces by $\{\Sigma_\alpha\}$. Then by Gauss--Bonnet, each $\{\Sigma_\alpha\}$ must be $\S^2$. 

Then, we can follow the proof of \cite[Theorem 1.10]{Balacheff2025}
to obtain that $M$ decomposes as a (possibly infinite) connected sum of closed prime 3-manifolds $\hat{P}_{ijk}$ modeled on some locally finite graph $\mathcal{G}$. We briefly sketch the steps here.

We first cut $M$ along the collection of spheres $\{\Sigma_\alpha\}$. The result consists of the connected
components $Y_{ij}$ of the pieces $Y_i = K_i \setminus K_{i-1}$. Denote by $\hat{Y}_{ij}$ the result of
capping the spherical boundary components of $Y_{ij}$ off by 3-balls. Then by the Kneser--Milnor decomposition theorem, each $\hat{Y}_{ij}$ decomposes as a connected sum of prime closed 3-manifolds $\hat{P}_{ijk}$. 
Next, we can construct a locally finite graph $\mathcal{G}$ as follows. The graph $\mathcal{G}$ has a vertex $v$ for each prime manifold $\hat{P}_{ijk}$, and there is an edge $e$ joining a pair of vertices $v, v'$ for each sphere in the common boundary $\partial P_v \cap \partial P_{v'}$, where $P_v$ and $P_{v'}$ are the prime manifolds corresponding to $v$ and $v'$. By construction, the manifold~$M$ decomposes as a (possibly infinite) connected sum of the closed prime 3-manifolds $\hat{P}_{ijk}$ modeled on $\mathcal{G}$.

Each closed prime 3-manifold $\hat{P}_{ijk}$ is either spherical, aspherical or diffeomorphic to $\S^2 \times \S^1$. However, since $M$ admits a metric of positive scalar curvature, by Theorem 1.7 of \cite{ccz2024aspherical}, none of the summands $\hat{P}_{ijk}$ can be aspherical. 
Therefore, $M$ is diffeomorphic to a (possibly infinite) connected sum of spherical manifolds and $\S^2 \times \S^1$ modeled on $\mathcal{G}$.

Optimality of the constant $C_3 = \frac{2}{3}$ follows from Proposition \ref{prop:construction}.
\end{proof}

\end{document}